\documentclass[10pt]{amsart}
\usepackage{amsmath, amssymb, amscd, mathtools, graphicx, amsfonts, enumerate, mathrsfs, xcolor}

\newtheorem{theorem}{Theorem}[section]
\newtheorem{proposition}[theorem]{Proposition}
\newtheorem{lemma}[theorem]{Lemma}
\newtheorem{corollary}[theorem]{Corollary}

\theoremstyle{definition}

\newtheorem{remark}[theorem]{Remark}

\numberwithin{equation}{section}

\newcommand{\N}{\mathbb{N}}                        
\newcommand{\K}{\mathbb{K}}                        
\newcommand{\R}{\mathbb{R}}                        
\newcommand{\C}{\mathbb{C}}                        
\newcommand{\OO}{\mathcal{O}}                      
\newcommand{\supp}{\mathrm{supp}}                  
\newcommand{\inexp}{\mathrm{exp}}                  
\newcommand{\NN}{\mathfrak{N}}                     
\newcommand{\LL}{\Lambda}
\newcommand{\mm}{\mathfrak{m}}                     
\newcommand{\nn}{\mathfrak{n}}                     
\newcommand{\tx}{\tilde{x}}
\newcommand{\ins}{\mathrm{in^*}}
\newcommand{\verts}{\mathcal{V}}

\begin{document}

\title{Finite determinacy and approximation of flat maps.}

\author{Aftab Patel}
\address{Department of Mathematics, The University of Western Ontario, London, Ontario, Canada N6A 5B7}
\email{apate378@uwo.ca}

\subjclass[2010]{32S05, 32S10, 32B99, 32C07, 14P15, 14P20, 13H10, 13C14}
\keywords{Cohen-Macaulay, singularity, Nash approximation, Hilbert-Samuel function, flatness, special fibre, tangent cone}

\begin{abstract}
	This paper considers the problems of finite determinacy and approximation of flat analytic maps from 
	germs of real or complex analytic spaces. It is shown that the flatness of analytic maps from germs of 
	real or complex analytic spaces whose local rings are Cohen-Macaulay is finitely determined. 
	Further, it is shown that flat maps from complete intersection and Cohen-Macaulay analytic germs can 
	be arbitrarily closely approximated by algebraic and Nash maps respectively in such a way 
	that the Hilbert-Samuel function of the special fibre is preserved. It is also proved that 
	in the complex case the preservation of the Hilbert-Samuel function implies the preservation 
	of Whitney's tangent cone. 
\end{abstract}
\maketitle


\section{Introduction}
\label{sec:intro}

In the study of the germs of real or complex analytic spaces and maps defined on them, it is of interest 
to know whether one can replace the convergent power series that describe them with finite polynomial
truncations or power series that are Nash, that is, which satisfy polynomial equations. 

J. Adamus and H. Seyedinejad in \cite{AS3} proved that the property of flatness of 
analytic morphisms from germs of analytic spaces whose local rings are complete intersections 
is finitely determined (\cite[Theorem 4.9]{AS3}). They also asked \cite[Question 4.10]{AS3} 
whether the complete intersection assumption on the domain can be relaxed. 
This paper provides an affirmative answer to this question: the flatness of 
mappings from analytic germs whose local rings are Cohen-Macaulay (henceforth called \emph{Cohen-Macaulay analytic germs}) 
is finitely determined. 
This is important because for complex analytic spaces the Cohen-Macaulay 
case is well understood geometrically. For  analytic maps from Cohen-Macaulay germs of complex analytic 
spaces to euclidean germs, flatness is equivalent to openness \cite[Proposition 3.20]{Fis}. 

Flat maps from germs of real or complex analytic spaces are central to deformation theory. In 
this context, it is of interest to know if analytic flat maps can be approximated by maps that 
are Nash in such a way that algebro-geometric properties of the special fibre are preserved. 
In this paper, it is shown that flat maps from analytic germs that are complete intersections 
can be polynomially approximated in such a 
way that the Hilbert-Samuel function of the special fibre is preserved.  
Further, it is shown that there exist Nash approximations of flat maps from Cohen-Macaulay analytic germs 
which preserve the Hilbert-Samuel function of the special fibre. The Hilbert-Samuel function encodes 
many algebro-geometric properties of analytic singularities and is 
considered a measure of singularity. It plays a central role in Hironaka's resolution of 
singularities (see \cite{BM2}), where it is used to 
track the progress of the desingularization. In this paper it is also shown that 
in the complex analytic case, the preservation of the Hilbert-Samuel function 
implies the preservation of Whitney's tangent cone \cite{Wh}. 

The results in this paper are proved using tools and techniques 
developed by the author and J. Adamus in \cite{JP} for proving the existence of 
equisingular algebraic approximations of analytic singularities. 

\subsection{Main Results}

Let $\K = \R$ or $\C$. Let $x = (x_1, \dots, x_n)$ and let $\mm$ denote the maximal ideal in 
the ring of convergent power series $\K\{x\}$. For a natural number $k\in \N$ and a power 
series $F \in \K\{x\}$ the \emph{$k$-jet} of $F$, denoted by $j^k F$, is the image of $F$ under 
the canonical epimorphism $\K\{x\} \rightarrow \K\{x\}/\mm^{k+1}$. For an $m$-tuple 
$\phi = (\phi_1, \dots, \phi_m) \in \K\{x\}^m$ let $j^k \phi = (j^k \phi_1, \dots, j^k \phi_m)$. 
An analytic mapping $\phi: X \rightarrow Y$ between analytic spaces $X$ and $Y$ is called \emph{flat
at $\xi \in X$} when the pullback homomorphism $\phi^*_{\xi}:\OO_{Y, \phi(\xi)} \rightarrow \OO_{X, \xi}$ makes 
$\OO_{X, \xi}$ into a flat $\OO_{Y, \phi(\xi)}$-module.

The first result that is proved in this paper is the following:  
\begin{theorem}
	\label{thm:cmfindet}
	Let $X$ be a $\K$-analytic subspace of $\K^n$. Suppose that $0 \in X$ and the local 
	ring $\OO_{X, 0}$ is Cohen-Macaulay. Also, let $\phi=(\phi_1, \dots, \phi_m): X \rightarrow \K^m$
	with $\phi(0) = 0$ be a $\K$-analytic mapping. Then there exists $\mu_0 \in \N$, 
	such that the following are equivalent:
	\begin{itemize}
		\item[(i)] $\phi$ is flat at zero. 
		\item[(ii)] For every $\mu \geq \mu_0$ every mapping $\psi=(\psi_1, \dots, \psi_m):X \rightarrow \K^m$ satisfying $j^{\mu} \phi = j^{\mu} \psi$ is flat at zero. 
		\item[(iii)] There exists $\mu \geq \mu_0$ such that every mapping $\psi=(\psi_1, \dots, \psi_m):X \rightarrow \K^m$ satisfying $j^{\mu} \phi = j^{\mu} \psi$ is flat at zero. 
	\end{itemize}
\end{theorem}

The above theorem is a generalization of 
\cite[Theorem 4.9]{AS3} and implies, in particular, that $j^{\mu}\phi$ defined by the polynomial 
truncation of $\phi$ is flat at zero for all $\mu \geq \mu_0$.

The next problems considered are the approximation of flat-maps from germs of analytic 
spaces whose local rings are complete intersections or Cohen-Macaulay. In the case 
of maps from complete intersections it is proved that any approximation that 
agrees with the defining power series of the original map upto to sufficiently high order
has a special fibre with the same Hilbert-Samuel function as that of the original (Theorem \ref{thm:ciapprox}).
This is a straightforward consequence of \cite[Theorem 4.9]{AS3} and \cite[Theorem 7.3]{JP}. 

A power 
series $F \in \K\{x\}$ is called an \emph{algebraic power series} or \emph{Nash} if it satisfies a non-trivial polynomial equation and 
the ring of such power series is denoted by $\K\langle x \rangle$. A map 
$\phi:X\rightarrow Y$ is called a \emph{Nash map} if it is defined by 
Nash power series. Recall that for an analytic germ $X_a$ the \emph{Hilbert-Samuel} 
function is defined as
\begin{equation}
	\label{eq:hsdef}
	H_{X_a} (\eta) = \dim_{\K} \OO_{X, a}/\mm_{X_a}^{\eta + 1} \qquad,\text{for all }\eta \in \N,
\end{equation}
where $\mm_{X_a}$ is the maximal ideal of $\OO_{X, a}$. 
If $\OO_{X, a} \cong \K\{x\}/I$, where $I$ is an ideal in $\K\{x\}$, the above is identified with the following
\begin{equation}
	\label{eq:hsdef2}
	H_{I} (\eta) = \dim_{\K} \K\{x\}/(I + \mm^{\eta + 1})\qquad,\text{for all }\eta \in \N.
\end{equation}
For flat analytic maps from Cohen-Macaulay analytic germs, there exist Nash approximations that 
are flat and whose special fibres have the same Hilbert-Samuel function as that of the 
special fibre of the original. Specifically,  
\begin{theorem}
\label{thm:cmapprox}
	Let $X$ be a $\K$-analytic subspace of $\K^n$. Suppose that $0 \in X$ and the local ring 
	$\OO_{X,0}$ is Cohen-Macaulay and that $\OO_{X,0} = \K\{x\}/I$ where 
	$I = (F_1, \dots, F_s)$. Also, let $\phi=(\phi_1, \dots, \phi_m): X \rightarrow \K^m$ 
	with $\phi(0) = 0$ be a $\K$-analytic mapping which is flat at 0. Then there exists $\mu_0 \in \N$, such that 
	for each $\mu \geq \mu_0$ there exist:  
	\begin{itemize}
		\item[(a)] A Nash, Cohen-Macaulay germ $\hat{X}_0 \subseteq \K^n$ with $\OO_{\hat{X}, 0} = \K\{x\}/\hat{I}$, where
			$\hat{I} = (G_{1}, \dots, G_{s})$, for some $G_{i} \in \K \langle x \rangle$ that satisfy
			$j^{\mu}G_{i} = j^{\mu} F_i$ for all $i$, such that $H_{\hat{I}} = H_{I}$.  
		\item[(b)] A Nash map 
			$\psi =(\psi_{1}, \dots \psi_{n}) : \hat{X} \rightarrow \K^m$ which is flat at 0 with 
			$j^{\mu} \phi_{i} = j^{\mu} \psi_{i}$ for all $i$, such that if 
			$J = (\phi_1, \dots, \phi_m) \subseteq \K\{x\}$ and 
			$\hat{J}=(\psi_1, \dots, \psi_m)\subseteq \K\{x\}$, then $H_{I + J} = H_{\hat{I} + \hat{J}}$. 
	\end{itemize}
\end{theorem}
Since the germ of special fibre $\phi^{-1}(0)_0$ (resp. $\psi^{-1}(0)_0$) is defined by the ideal $I + J$ (resp. $\hat{I} + \hat{J}$) in $\K \{ x\}$, 
the above implies that the Hilbert-Samuel function of the special fibre of $\phi$ and its approximation $\psi$ coincide. 

It is also shown that in the case when $\K = \C$, the preservation of the Hilbert-Samuel function 
in the approximation results of this paper implies the preservation of Whitney's tangent cone 
(Corollaries \ref{cor:cicones} and \ref{cor:cmcones}). 
\subsection{Structure of Paper}
\label{sec:paper-struct}

For the proofs of Theorems \ref{thm:cmfindet} and  
\ref{thm:cmapprox} and their Corollaries \ref{cor:cicones} and \ref{cor:cmcones},
four main tools are used. 
These are:
\begin{itemize}
	\item[(i)] Hironaka's Diagram of Initial Exponents, which is a combinatorial object 
		related to an ideal in a convergent power series ring (Section \ref{sec:diag}).
		The diagram of initial exponents of an ideal completely determines its Hilbert-Samuel
		function (Lemma \ref{lem:HS-diagram-complement}). Stability of the diagram, 
		in an appropriate sense, is equivalent to flatness  
		(Propositions \ref{prop:hirflat2} and \ref{prop:flatness-diagram}). 
	\item[(ii)] Becker's notion of standard bases for ideals in power series rings, and
		Becker's s-series criterion \cite{Be1, Be2}, which is a criterion 
		analogous to 
		Buchberger's criterion for Gr\"{o}bner bases of ideals in polynomial rings (Section 
		\ref{subsec:sb}). This is used to encode the shape of the diagram of 
		initial exponents in a system of polynomial equations in the proof of Theorem \ref{thm:cmapprox}. 
	\item[(iii)] Artin's Approximation theorem \cite[Theorem 1.10]{Ar2}. This is 
		used in the proof of Theorem \ref{thm:cmapprox} to obtain the 
		required Nash approximations. 
	\item[(iv)] The fact that the quotient of a Cohen-Macaulay ring by a regular sequence 
		is Cohen-Macaulay \cite[Corollary B.8.3]{GLS} and a well known flatness 
		criterion \cite[Theorem B.8.11]{GLS} that relates the 
		flatness of a map from a Cohen-Macaulay germ to the dimension of the 
		special fibre. 
	\end{itemize}
This paper begins by presenting definitions and results related to tools (i) and (ii) that 
will be needed in the proofs that come later.
After developing the tools, the proofs of Theorems \ref{thm:cmfindet}, 
\ref{thm:ciapprox}, \ref{thm:cmapprox} and their corollaries are 
presented. 
Throughout this paper the germ of a real or complex analytic space is referred to as 
an \emph{analytic germ}. 

\subsection*{Acknowledgement} The author would like to acknowledge the 
feedback of Janusz Adamus, which helped greatly in improving the quality of the 
writing in this paper. 

\section{Diagrams of initial exponents and flatness}
\label{sec:diag}

Let $x = (x_1, \dots, x_n)$ and $y = (y_1, \dots, y_m)$. Further, let $A = \K\{y\}$ or $\K$. 
A function on $\K^n$, $\LL(\beta)=\sum_{i=1}^n\lambda_i\beta_i$, for some $\lambda_i>0$,  
is called a  \emph{positive linear form} on $\K^n$. For each such $\LL$
the lexicographic ordering of the $n+1$-tuples $(\LL(\beta), \beta_1, \dots, \beta_n)$ 
defines a total ordering on $\N^n$. We call this the \emph{ordering induced by $\LL$}.  
A power series $F \in A\{x\}$ may be written as $F = \sum_{\beta \in \N^n} f_{\beta} x^{\beta}$ 
where $x^{\beta} = x_1^{\beta_1} \dots x_n^{\beta_n}$. The set $\supp(F) = \{ \beta \in \N^n : f_{\beta} \neq 0\}$ is 
called the \emph{support} of $F$.  
The \emph{initial exponent of $F$ with 
respect to $\LL$} is $\inexp_{\LL}(F) = \min_{\LL} \{ \beta\in \supp(F)\}$, where the minimum is taken 
with respect to ordering induced by $\LL$.  Given an ideal $I \in A\{x\}$, 
\begin{equation}
	\NN_{\LL}(I) = \{\inexp_{\LL}(F) : F \in I \setminus \{0 \}\},
\end{equation}
is called the \emph{diagram of initial exponents} of $I$ relative to $\LL$. When
$\LL(\beta) = |\beta| = \beta_1 + \dots + \beta_n$, the subscript is 
suppressed and the diagram of initial exponents is denoted simply by $\NN(I)$. 

If $A = \K\{y\}$, the mapping $ A \ni F(y) \mapsto F(0) \in \K$ of the 
evaluation of the $y$ variables at $0$ induces an evaluation mapping, 
\begin{equation*}
	A\{x\} \ni F = \sum_{\beta \in \N^n} f_{\beta} (y) x^{\beta} \mapsto F(0) = \sum_{\beta \in \N^n} 
	f_{\beta} (0) x^{\beta} \in \K\{x\}. 
\end{equation*}
(In case $A=\K$, this is just the identity mapping). For an ideal $I \subseteq A\{x\}$, 
$I(0) = \{F(0) : F \in I\} \subseteq \K\{x\}$ is called the \emph{evaluated ideal}. 

\begin{remark}
	\label{rem:vertices}
	For every ideal $I$ in $A\{x\}$ and for every
	positive linear form $\LL$, there exists a unique smallest (finite) set
	$\verts_{\LL}(I)\subset\NN_\LL(I)$ such that $\verts_{\LL}(I)+\N^k=\NN_{\LL}(I)$ (see,
	e.g., \cite[Lemma 3.8]{BM1}). The elements of $\verts_{\LL}(I)$ are called the
	\emph{vertices} of the diagram $\NN_\LL(I)$. When $\LL(\beta) = |\beta|$, the 
	subscript is suppressed and the set of vertices is denoted by $\verts (I)$. 
\end{remark}

For a positive linear form $\LL$ and $\mu \in \N$, let $\nn_{\LL,\mu}$ be the ideal in
$\K\{x\}$ generated by all the monomials $x^\beta=x_1^{\beta_1}\dots
x_n^{\beta_n}$ with $\LL(\beta)\geq\mu$. (Note that, by positivity of the linear
form $\LL$, the ideals $\nn_{\LL,\mu}$ are $\mm$-primary for every $\mu$.
Moreover, for $\LL(\beta)=|\beta|$ we have $\nn_{\LL,\mu}=\mm^\mu$.)
By the following lemma, diagram of initial exponents of an ideal determines its Hilbert-Samuel function. 
\begin{lemma}{\cite[Lemma 6.2]{JP}}
	\label{lem:HS-diagram-complement}
	Let $\lambda_1,\dots,\lambda_n>0$ be arbitrary, and let
	$\LL(\beta)=\sum_{j=1}^n\lambda_j\beta_j$. Then, for any ideal $I$ in
	$\K\{x\}$ and for every $\eta\geq1$,
	\[
		\#\{\beta\in\N^n\setminus\NN_\LL(I):\LL(\beta)\leq\eta\}\ =\ \dim_\K\frac{\K\{x\}}{I+\nn_{\LL,\eta+1}}\,,
	\]
	where the dimension on the right side is in the sense of $\K$-vector spaces.
	In particular, the Hilbert-Samuel function $H_I$ of \,$\K\{x\}/I$ satisfies
	\[
		H_I(\eta)=\#\{\beta\in\N^n\setminus\NN(I):|\beta|\leq\eta\},\quad\mathrm{for\ all\ }\eta\geq1\,.
	\]
\end{lemma}

The following proposition proved in \cite{JP}, which relates the 
diagram of $I$ to the dimension of $\K\{x\}/I$, will be 
used in the proofs of Theorem \ref{thm:cmfindet} and Theorem \ref{thm:cmapprox}. 
\begin{proposition}{\cite[Proposition 7.1]{JP}}
\label{prop:diagram-dim}
For an ideal $I$ in $\K\{x\}$, the following conditions are equivalent:
\begin{itemize}
\item[(i)] $\dim(\K\{x\}/I)\leq\dim\K\{x\}-k$.
\item[(ii)] After a generic linear change of coordinates in $\K^n$, the diagram $\NN(I)$ has a vertex on each of the first $k$ coordinate axes in $\N^n$.
\end{itemize}
\end{proposition}

\begin{remark}
\label{rem:CM-flat}
	Let $I$ be a proper ideal in $\K\{x\}$ with $\dim\K\{x\}/I=n-k$, and suppose 
	that $\K\{x\}/I$ is a finite $\K\{\tx\}$-module, where $\tx=(x_{k+1},\dots,x_n)$. 
	Then, $\K\{x\}/I$ is Cohen-Macaulay if and only if it is a flat $\K\{\tx\}$-module.
	This follows from the flatness criterion \cite[Theorem B.8.11]{GLS}.
\end{remark}

\subsection{Diagrams and flatness}
\label{subsec:diagflat}

For an ideal $I \in A\{x\}$, let $\Delta = \N^n \setminus \NN(I(0))$, and 
$A\{x\}^{\Delta} = \{F \in A\{x\} : \supp(F) \subset \Delta\}$. Further, 
let $\kappa$ be the restriction of the canonical projection $A\{x\} \rightarrow
A \{x\}/I$ to $A\{x\}^{\Delta}$. The following flatness criterion due to 
Hironaka will be used in the proof of Theorem \ref{thm:cmfindet}. 

\begin{proposition}{\cite[\S6, Proposition 10]{H}}
	\label{prop:hirflat2}
	The ring $A\{x\}/I$ is flat as an $A$-module if and only if $\kappa$ is 
	bijective. 
\end{proposition}
The following proposition will also be used in the proof of Theorem \ref{thm:cmapprox}.  
\begin{proposition}{\cite[Proposition 2.5]{JP}}
	\label{prop:flatness-diagram}
	Let $I$ be an ideal in $\K\{x\}$, let $1\leq k<n$, and let $\tx$ denote
	the variables $(x_{k+1},\dots,x_n)$.
	\begin{itemize}
		\item[(i)] If there exist a positive linear form $\LL$ on $\K^n$
			and a set $D\subset\N^k$ such that
			$\NN_\LL(I)=D\times\N^{n-k}$\!, then $\K\{x\}/I$ is a
			flat $\K\{\tx\}$-module.
		\item[(ii)] If $\K\{x\}/I$ is a flat $\K\{\tx\}$-module, then
			there exist $l_0\in\N$ and a set $D\subset\N^k$ such
			that, for all $l\geq l_0$, the diagram $\NN_\LL(I)$ with
			respect to the linear form
			$\displaystyle{\LL(\beta)=\sum_{i=1}^k\beta_i+\sum_{j=k+1}^n\!l\beta_j}$
			satisfies $\NN_\LL(I)=D\times\N^{n-k}$.
	\end{itemize}
\end{proposition}

\subsection{Approximation of Diagrams}
Let $\Lambda(\beta) = \sum_{i = 1}^{n} \lambda_i \beta_i$ be a positive linear form on $\K^n$. 
The notion of jets introduced in Section \ref{sec:intro} can be generalized as
follows: For a natural number $\mu\in\N$ and a power series $F\in\K\{x\}$, the
\emph{$\mu$-jet of $F$ with respect to $\LL$}, denoted $j_\LL^\mu(F)$, is the
image of $F$ under the canonical epimorphism
$\K\{x\}\to\K\{x\}/\nn_{\LL,\mu+1}$, where $\nn_{\LL, \mu+1}$ is defined as in 
the beginning of Section \ref{sec:diag}. 
Note that the notion of $\mu$-jets defined in Section \ref{sec:intro} are obtained by 
applying this definition with
$\LL(\beta)=|\beta|=\beta_1+\dots+\beta_n$. The following lemma relates 
the diagram of an ideal to its approximation. 

For an ideal $I = (F_1, \dots, F_k) \subseteq \K\{x\}$, a positive linear 
form $\Lambda$ and $\mu \geq 1$, let
\begin{equation*}
	U_{\Lambda}^{\mu}(I) = \{(G_1, \dots, G_k) \subseteq \K\{x\} : 
		j^{\mu}_{\Lambda} (G_i) = j^{\mu}_{\Lambda} (F_i), 1\leq i \leq k\}  
\end{equation*}

\begin{lemma}{\cite[Lemma 6.5]{JP}}
	\label{lem:diagram-up-to-l}
	Let $I$ be an ideal in $\K\{x\}$ and let $\LL$ be a positive linear form on $\K^n$. Let $l_0=\max\{\LL(\beta^i):1\leq i\leq t\}$, where $\beta^1,\dots,\beta^t$ are the vertices of the diagram $\NN_\LL(I)$. Then:
	\begin{itemize}
		\item[(i)] For every $\mu\geq l_0$ and $I_\mu\in U_\LL^\mu(I)$, we have $\NN_\LL(I_\mu)\supset\NN_\LL(I)$.
		\item[(ii)] Given $l\geq l_0$, for every $\mu\geq l$ and $I_\mu\in U_\LL^\mu(I)$, we have
		      \[
		      	\NN_\LL(I_\mu)\cap\{\beta\in\N^n:\LL(\beta)\leq l\}=\NN_\LL(I)\cap\{\beta\in\N^n:\LL(\beta)\leq l\}\,.
		      \]
	\end{itemize}
\end{lemma}

The following is an immediate consequence of the above lemma. 
\begin{lemma}
	\label{lem:eqdiags}
	If $I = (F_1, \dots, F_k)$ is an ideal in $\K\{x\}$, 
	$\mu$ is an integer such that $\mu > \max\{|\beta|: \beta \in \verts(I)\}$, and 
	$J \in U^{\mu}(I)$, then $H_I = H_J$ implies that $\NN(I) = \NN(J)$. 
\end{lemma}

\begin{proof}
	Let $D_{\eta} = \{\beta\in\N^n:|\beta|\leq\eta\}$.
	Lemma \ref{lem:diagram-up-to-l} implies that $\NN(I) \subseteq \NN(J)$.
	Suppose that $\NN(I) \subsetneq \NN(J)$.
	Then, $\N^n \setminus \NN(J) \subsetneq \N^n \setminus \NN(I)$, 
	which implies that there exists 
	a $\beta \in \NN(J)$ such that $\beta \in \N^n \setminus \NN(I)$. Let 
	$\eta_0 \in \N$ be such that $|\beta| \leq \eta_0$. 
	Then $D_{\eta_0} \cap (\N^n \setminus \NN(J)) \subsetneq D_{\eta_0} \cap (\N^n \setminus \NN(I))$
	and by Lemma \ref{lem:HS-diagram-complement}, 
	$H_J(\eta_0) = \#(D_{\eta_0} \cap (\N^n \setminus \NN(J))) <  \#(D_{\eta_0} \cap (\N^n \setminus \NN(I))) = 
	H_I(\eta_0)$, which contradicts $H_I = H_J$.   
	Therefore $\NN(I) = \NN(J)$. 
\end{proof}

\subsection{Standard bases}
\label{subsec:sb}
Let again $\LL(\beta)=\sum_{j=1}^n\lambda_j\beta_j$ be a positive linear form on
$\K^n$, and let $\N^n$ be given the total ordering defined by the lexicographic
ordering of the $(n+1)$-tuples $(\LL(\beta),\beta_n,\dots,\beta_1)$.
For $F\in\K\{x\}$, let as before $\inexp_\LL{F}={\min}_\LL\{\beta\in\supp{F}\}$
denote the initial exponent of $F$ relative to $\LL$.

Let $I$ be an ideal in $\K\{x\}$.
A collection $S=\{G_1,\dots,G_t\}\subset I$ forms a \emph{standard basis} of $I$
(relative to $\LL$), when for every $F\in I \setminus \{ 0 \}$ there exists $i\in\{1,\dots,t\}$
such that $\inexp_\LL{F}\in\inexp_\LL{G_i}+\N^n$.

\begin{remark}
	\label{rem:st-bases}
	\begin{itemize}
		\item[(1)] It follows directly from definition that every
			standard basis $S$ of $I$ relative to $\LL$ contains
			representatives of all the vertices of the diagram
			$\NN_\LL(I)$ (that is, for every vertex $\beta^i$ of
			$\NN_\LL(I)$ there exists $G_i\in S$ with
			$\beta^i=\inexp_\LL{G_i}$). Also, it is a consequence \cite[Corollary 3.9]{BM1} of 
			Hironaka's Division Theorem \cite[Theorem 3.1, 3.4]{BM1}, 
			that a standard basis of $I$ generates $I$.  
		\item[(2)] Note that the term ``standard basis'' in most of
			modern literature refers to a collection defined by more
			restrictive conditions than the one above (see, e.g.,
			\cite[Corollary 3.9]{BM1} or \cite[Corollary 3.19]{BM2}). In
			particular, standard bases as defined 
			in this paper are not unique and may
			contain elements which do not represent vertices of the
			diagram.
	\end{itemize}
\end{remark}

For any $F=\sum_\beta f_\beta x^\beta$ and $G=\sum_\beta g_\beta x^\beta$ in
$\K\{x\}$, one defines their \emph{s-series} $S(F,G)$ with respect to $\LL$ as
follows: If $\beta_F=\inexp_\LL{F}$, $\beta_G=\inexp_\LL{G}$, and
$x^\gamma=\mathrm{lcm}(x^{\beta_F},x^{\beta_G})$, then
\[
	S(F,G)\coloneqq g_{\beta_G}x^{\gamma-\beta_F}\!\cdot F - f_{\beta_F}x^{\gamma-\beta_G}\!\cdot G\,.
\]
Given $G_1,\dots,G_t,F\in\K\{x\}$, $F$ has a \emph{standard
representation} in terms of $\{G_1,\dots,G_t\}$ with respect to $\LL$, when
there exist $Q_1,\dots,Q_t\in\K\{x\}$ such that
\[
	F=\sum_{i=1}^tQ_iG_i\qquad\mathrm{and}\qquad \inexp_\LL{F}\leq\min\{\inexp_\LL(Q_iG_i):i=1,\dots,t\}\,.
\]
Here, by convention $\inexp_\LL{F}<\inexp_\LL{0}$, for any $\LL$
and any non-zero $F$. 

\begin{remark}
	\label{rem:ssrep}
	It follows from Hironaka's Division theorem \cite[Theorems 3.1, 3.4]{BM1} that 
	if $F_1, \dots, F_t$ form a standard basis of 
	an ideal $I \subseteq \K\{x\}$ then every 
	$G \in I$ has a standard representation in 
	terms of $F_1, \dots, F_t$.  
\end{remark}

The following Theorem proved by Becker in \cite{Be1, Be2} gives a criterion 
for determining whether a collection of power series forms a standard basis. 
\begin{theorem}[{\cite[Theorem 4.1]{Be1}}]
	\label{thm:Becker}
	Let $S$ be a finite subset of $\K\{x\}$. Then, $S$ is a standard basis
	(relative to $\LL$) of the ideal it generates if and only if for any
	$G_1,G_2\in S$ the s-series $S(G_1,G_2)$ has a standard representation
	in terms of $S$.
\end{theorem}

\section{Finite determinacy of flat maps}
\label{sec:fd}

\subsection{Proof of Theorem \ref{thm:cmfindet}}
	Suppose that $\dim \OO_{X, 0} = n - k$, and
	$\OO_{X,0} \cong \K\{x\}/I$ where $I = (F_1, \dots, F_s)$.
	Further, let $J$ be the ideal in 
	$\K\{x\}$ generated by $\phi_1, \dots, \phi_m$. 
	Suppose $\phi$ is flat at zero.
	Then 
	$\phi_1, \dots, \phi_m$ form an $\OO_{X, 0}$-regular sequence 
	by \cite[Theorem B.8.11]{GLS}, and hence $\dim \K\{x\}/(I + J) = n - k - m$.  
	Up to a generic linear change of coordinates, by Proposition \ref{prop:diagram-dim}, the latter 
	equivalent to saying that 
	$\NN(I + J)$ has a vertex on each of the first $k+m$ coordinate axes. 
	Let $G_1, \dots, G_{k+m} \in I + J$ be representatives 
	of these vertices. There are $Q^{i}_{p} \in \K\{x\}$ for $1\leq i \leq m$, $1\leq p\leq k+m$ such that
	\begin{equation*}
		G_i = \sum_{p = 1}^{s} Q^{i}_{p} F_p + \sum_{q = 1}^{m} Q^{i}_{s + q} \phi_q, 
		\text{ for } i = 1, \dots, k+m.
	\end{equation*}
	Fix $\mu' \geq \max\{|\beta|:\beta \in \verts(I+J)\}$.
	Then for any $\mu \geq \mu'$, 	
	\begin{equation}
		\label{eq:findet2}
		\inexp(j^{\mu} G_i) = \inexp(G_i), \text{ for } i = 1, \dots,k+m.  
	\end{equation}
	Now suppose that $\psi = (\psi_1, \dots, \psi_m):X \rightarrow \K^m$ is such that 
	$j^{\mu}\psi_q = j^{\mu} \phi_q$ for $q = 1, \dots, m$ and 
	set $\hat{J} = (\psi_1, \dots, \psi_m) \subseteq \K\{x\}$. Then, 
	for $i = 1, \dots, k+m$, 
	\begin{eqnarray*}
		j^{\mu} (\sum_{p = 1}^{s} Q^{i}_{p} F_p + 
		\sum_{q = 1}^{m} Q^{i}_{s + q} \psi_q) 
			    &=& j^{\mu}  (\sum_{p = 1}^{s} Q^{i}_{p} F_p + 
		\sum_{q = 1}^{m} Q^{i}_{s + q} j^{\mu}\psi_q) \\
			    &=& j^{\mu} (\sum_{p = 1}^{s} Q^{i}_{p} F_p + 
		\sum_{q = 1}^{m} Q^{i}_{s + q} j^{\mu}\phi_q) \\
			    &=& j^{\mu} (\sum_{p = 1}^{s} Q^{i}_{p} F_p + 
		\sum_{q = 1}^{m} Q^{i}_{s + q} \phi_q) \\
			    &=& j^{\mu} G_i
	\end{eqnarray*}
	Hence, by \eqref{eq:findet2}, $\NN(I + \hat{J})$ has a vertex on each of the 
	first $k+m$ coordinate axes. By Proposition \ref{prop:diagram-dim} again, 
	\begin{equation*}
		\dim \K\{x\}/(I = \hat{J}) \leq n-k-m.
	\end{equation*}
	But $\dim \K\{x\}/(I + \hat{J}) \geq (n - k) - m$ because 
	$\hat{J}$ has $m$ generators, therefore, 
	$\dim \K\{x\}/(I + \hat{J}) = n - k - m$. Then, by \cite[Theorem B.8.11]{GLS},
	$\psi = (\psi_1, \dots, \psi_m)$ is 
	flat at zero. Therefore, (i) implies (ii) if $\mu \geq \mu'$. The implication (ii) $\implies$ (iii) is trivial. 

	For (iii) implies (i), $\K\{x\}/I$ may be identified with the 
	graph of $\phi$. That is consider $\K\{x, y\}/L$ where $y = (y_1, \dots, y_m)$ and 
	$L = (F_1, \dots, F_s, y_1 - \phi_1(x), \dots, y_m - \phi_m(x))$. Now, suppose that 
	$\phi$ is not flat at zero. Then by Proposition \ref{prop:hirflat2} there is 
	a non-zero $R \in L$ supported outside of 
	$\NN(L(0))$, where the evaluation is at $y = 0$.  
	Write $R = \sum_{i = 1}^{s} H_iF_i + \sum_{p = 1}^{m} K_p (y_p - \phi_p)$ for 
	some $H_i,K_p \in \K\{x\}$, $1\leq i \leq s$, $1 \leq p \leq m$.  
	Let $\mu'' \geq \max\{l_0, |\inexp(R)|\}$, where 
	$l_0 = \max\{|\beta|:\beta \in \verts(L(0))\}$. Now, for 
	$\mu \geq \mu''$, and a map 
	$\psi=(\psi_1, \dots, \psi_m)$ where 
	$j^\mu \psi_i = j^\mu \phi_i$ for $1\leq i \leq m$, 
	let $\tilde{R} = \sum_{i = 1}^{s} H_iF_i + \sum_{k = 1}^{m} K_k (y_k - \psi_k)$. 
	Then, $\tilde{R} \in \hat{L}$ where $\hat{L} = (F_1, \dots, F_s, y_1 - \psi_1(x), \dots, y_m - \psi_m(x))$. 
	By the choice of $\mu$, $\inexp(R) = \inexp(\tilde{R})$. By Lemma \ref{lem:diagram-up-to-l}, 
	$\inexp(\tilde{R}) \notin \NN(\hat{L}(0))$, which implies that $\psi$ is not 
	flat by Propositions \ref{prop:hirflat2}. The 
	theorem holds for $\mu_0 \geq \max\{\mu', \mu''\}$.  
	\qed.

By \cite[Theorem B.8.11]{AS3} the flatness of $\phi = (\phi_1 \dots, \phi_m):X \rightarrow \K^m$
at zero is equivalent to $\phi_1, \dots, \phi_m$ forming an $\OO_{X, 0}$-regular sequence. 
Therefore, the following is an immediate consequence of Theorem \ref{thm:cmfindet}. 
\begin{corollary}
	\label{cor:regseqfindet}
	Suppose that $A = \K\{x\}/I$ is a Cohen-Macaulay ring 
	and that $\phi_1, \dots, \phi_m$ is a sequence of elements of $A$. 
	Then there exists $\mu_0 \in \N$ such that the following are equivalent:
	\begin{itemize}
		\item[(i)] $\phi_1, \dots, \phi_m$ is a regular sequence. 
		\item[(ii)] For all $\mu \geq \mu_0$ any sequence $\psi_1, \dots, \psi_m \in A$ satisfying
			$j^{\mu}\phi_i = j^{\mu}\psi_i$ for all $i$ is a regular sequence. 
		\item[(iii)] There exists $\mu \geq \mu_0$ such that any sequence $\psi_1, \dots, \psi_m \in A$ satisfying
			$j^{\mu}\phi_i = j^{\mu}\psi_i$ for all $i$ is a regular sequence.
	\end{itemize}
\end{corollary}

\section{Approximation of flat maps}
\label{sec:approx}

\subsection{Flat maps from complete intersections}

\begin{theorem}
	\label{thm:ciapprox}
	Let $X$ be a $\K$-analytic subspace of $\K^n$. Suppose that $0 \in X$ and 
	the local ring $\OO_{X, 0}$ is a complete intersection. Also, let 
	$\phi = (\phi_1, \dots, \phi_m): X \rightarrow \K^n$ be a $\K$-analytic mapping
	that is flat at zero. Then there exists $\mu_0 \in \N$ such that for each 
	$\mu \geq \mu_0$, every map $\psi = (\psi_1, \dots, \psi_m): X \rightarrow \K^n$
	that satisfies $j^{\mu}\psi = j^{\mu}\phi$ has the following properties: 
	\begin{itemize}
		\item[(i)] $\psi$ is flat at zero.
		\item[(ii)] $H_{\phi^{-1}(0),0} = H_{\psi^{-1}(0),0}$. 
	\end{itemize}
\end{theorem}
\begin{proof}
	Suppose that $\OO_{X,0} \cong \K\{x\}$ where $I = (F_1, \dots, F_k)$ for some 
	$F_1, \dots, F_k \in \K\{x\}$ that form a regular sequence in $\K\{x\}$. By 
	\cite[Theorem B.8.11]{GLS} the flatness of $\phi$ implies that 
	$\phi_1, \dots, \phi_m$ form an $\OO_{X,0}$-regular sequence.
	This is equivalent to saying that, $F_1, \dots, F_k, \phi_1, \dots, \phi_m$ form  
	a regular sequence in $\K\{x\}$. Therefore 
	the ideal $J = (F_1, \dots, F_k, \phi_1, \dots, \phi_m)$ is 
	a complete intersection ideal. Further, 
	the local ring of the special fibre of $\phi$, satisfies 
	$\OO_{\phi^{-1}(0),0} \cong \K\{x\}/J$. By 
	\cite[Corollary 4.8]{AS3} there exists  $\mu' \in \N$ 
	such that for each $\mu \geq \mu'$ and 
	any $\psi_1, \dots, \psi_m$ that satisfy $j^{\mu}\psi_i = j^{\mu}\phi_i$ for 
	$1 \leq i \leq m$, $F_1, \dots, F_k, \psi_1, \dots, \psi_m$ form 
	a regular sequence in $\K\{x\}$. This implies that the 
	map $\psi = (\psi_1, \dots, \psi_m):X\rightarrow \K^m$ is 
	flat at zero. 
	By \cite[Theorem 7.3]{JP} there exists $\mu'' \in \N$ such that for 
	all $\mu \geq \mu''$, $J_{\mu} = (F_1, \dots, F_k, \psi_1, \dots, \psi_m)$ 
	is a complete intersection and $H_{J_{\mu}} = H_J$. The 
	theorem thus holds for $\mu_0 \geq \max\{\mu', \mu''\}$. 
\end{proof}

\subsection{Flat maps from Cohen-Macaulay analytic germs}
\subsection*{Proof of Theorem \ref{thm:cmapprox}}
	Suppose that $\OO_{X,0} = \K\{x\}/I$ and $I = (F_1, \dots, F_s)$. 
	By Proposition~\ref{prop:diagram-dim}, after a generic linear change of
	coordinates in $\K^n$, the diagram $\NN(I)$ has a vertex on each of the
	first $k$ coordinate axes in $\N^n$. It follows that $\K\{x\}/I$ is
	$\K\{\tx\}$-finite, where $\tx = (x_{k+1}, \dots, x_n)$ and hence $\K\{\tx\}$-flat 
	because it is Cohen-Macaulay (Remark~\ref{rem:CM-flat}).
	Therefore, by Proposition~\ref{prop:flatness-diagram} there exists $l_1\geq1$ such that for the linear form
	\begin{equation*}
		\LL_1(\beta)=\sum_{i=1}^k\beta_i+\sum_{j=k+1}^n\!\!l_1\beta_j\,,
	\end{equation*}
	one has, 
	\begin{equation}
		\label{eq:orderdiag1}
		\NN_{\LL_1}(I)=D_1\times\N^{n-k}\,\text{ for some }D_1 \subset \N^{k}.
	\end{equation}
	The set $\{F_1, \dots, F_s\}$ may be extended by power series $F_{s+1},\dots,$ $F_r\in I$ such that
	the collection $\{F_1,\dots,F_r\}$
	contains representatives of all the vertices $\NN_{\LL_1}(I)$. Since $I$ is
	generated by $\{F_1,\dots,F_s\}$, there are $H_p^q\in\K\{x\}$ such that
	\begin{equation*}
		F_{s+p}=\sum_{q=1}^sH_p^qF_q\,,\qquad p=1,\dots,r-s.
	\end{equation*}
	Then, $\{F_1,\dots,F_r\}$ is a set of generators of $I$ and a standard basis of
	$I$ relative to $\LL_1$ by Remark \ref{rem:st-bases}. 
	For
	$i,j\in\{1,\dots,r\}$, $i<j$, let $S_{i,j}=S(F_i,F_j)$ denote the s-series
	of the pair $(F_i,F_j)$.
	By Theorem~\ref{thm:Becker}, there exist $Q^{i,j}_m\in\K\{x\}$, $i,j,m\in\{1,\dots,r\}$, 
	such that
	\begin{equation*}	
		S_{i,j}=\sum_{m=1}^rQ^{i,j}_mF_m\qquad\mathrm{and}\qquad \inexp_{\LL_1}{S_{i,j}}\leq\min\{\inexp_{\LL_1}(Q^{i,j}_mF_m):m=1,\dots,r\}\,.
	\end{equation*} 
	Let $J = (\phi_1, \dots, \phi_m) \subseteq \K\{x\}$. The flatness of $\phi$ implies that the representatives of $\phi_1, \dots, \phi_m$
	form a regular sequence in $\K\{x\}/I$. 

	By the fact that a quotient of a Cohen-Macaulay ring by a regular sequence is Cohen-Macaulay \cite[Corollary B.8.3]{GLS}, the ring $\K\{x\}/(I + J)$ is Cohen-Macaulay with 
	dimension $\dim (\K\{x\}/I) - m = n-k-m$. Therefore, as above, after a generic linear change of coordinates in $\K^n$ there exists 
	$l_2 \geq 1$ such that for the linear form 
	\begin{equation*}
		\LL_2(\beta) = \sum_{i=1}^{k+m} \beta_i + \sum_{j=k+m+1}^n l_2\beta_j\,,
	\end{equation*}
        one has 	
	\begin{equation}
		\label{eq:orderdiag2}
		\NN_{\LL_2}(I+J) = D_2 \times \N^{n-k-m}\, \text{ for some }D_2 \subset \N^{k+m}.
	\end{equation}
	Since \emph{generic} linear changes of coordinates form an open dense subset of the space of 
	all linear changes of coordinates, and both the coordinate changes made above  
	are generic, one can, in fact, choose one change of coordinates for which both \eqref{eq:orderdiag1} and 
	\eqref{eq:orderdiag2} hold.

	One has $I + J = (F_1, \dots, F_s, \phi_1, \dots, \phi_m)$. This set of 
	generators may be extended by $\phi_{m+1}, \dots, \phi_{l}$, such that 
	$F_1, \dots, F_s, \phi_1, \dots, \phi_l$ form a standard basis for $I + J$ with respect to $\LL_2$. 
	We then have $\bar{H}_p^q \in \K\{x\}$ such that
	\begin{equation*}
		\phi_{m+p} = \sum_{q = 1}^s \bar{H}_p^q F_q + \sum_{q = 1}^{m} \bar{H}_p^{q+s} \phi_q \,,\qquad p = 1, \dots, l-m.
	\end{equation*}
	Let $S_{i,j}^{(1)} = S(F_i, F_j)$ for $1 \leq i < j \leq s$,
	$S_{i,j}^{(2)} = S(F_i, \phi_j)$ for $i \in \{1, \dots, s\}$, $j \in \{1, \dots, l\}$, and 
	$S_{i,j}^{(3)} = S(\phi_i, \phi_j)$, for $1 \leq i < j \leq l$ be the s-series with respect to the ordering induced by $\LL_2$. 
	By Theorem~\ref{thm:Becker} there exist: 
	$\tilde{Q}^{i, j}_m, \bar{Q}^{i, j}_m, \hat{Q}^{i, j}_m \in \K\{x\}$ with index ranges
	$1 \leq i < j \leq s$, $i \in \{1, \dots, s\}$, $j \in \{1, \dots, l\}$, and $1 \leq i < j \leq l$ respectively, such that
	\begin{equation*}
		S_{i,j}^{(1)} = \sum_{m = 1}^{s} \tilde{Q}^{i,j}_m F_m + \sum_{m = 1}^{l} \tilde{Q}^{i,j}_{m + s} \phi_m
	\end{equation*}
	\begin{equation*}
		S_{i,j}^{(2)} = \sum_{m = 1}^{s} \bar{Q}^{i,j}_m F_m + \sum_{m = 1}^{l} \bar{Q}^{i,j}_{m + s} \phi_m
	\end{equation*}
	and
	\begin{equation*}
		S_{i,j}^{(3)} = \sum_{m = 1}^{s} \hat{Q}^{i,j}_m F_m + \sum_{m = 1}^{l} \hat{Q}^{i,j}_{m + s} \phi_m
	\end{equation*}
	In the above equations $\inexp_{\LL_2} S_{i,j}^{(1)} \leq \min\{e_1, e_2\}$, $\inexp_{\LL_2} S_{i,j}^{(2)} \leq \min\{f_1, f_2\}$, and 
	$\inexp_{\LL_2} S_{i,j}^{(3)} \leq \min\{g_1, g_2\}$ where, 
	\begin{eqnarray*}
		e_1 &=& \min\{\inexp_{\LL_2}(\tilde{Q}^{i,j}_mF_m):m=1,\dots,s\},\\
		e_2 &=& \min\{\inexp_{\LL_2}(\tilde{Q}^{i,j}_{m+r}\phi_m):m=1,\dots,l\},\\
		f_1 &=& \min\{\inexp_{\LL_2}(\bar{Q}^{i,j}_mF_m):m=1,\dots,s\},\\
		f_2 &=& \min\{\inexp_{\LL_2}(\bar{Q}^{i,j}_{m+r}\phi_m):m=1,\dots,l\},\\
		g_1 &=& \min\{\inexp_{\LL_2}(\hat{Q}^{i,j}_mF_m):m=1,\dots,s\},\\
		g_2 &=& \min\{\inexp_{\LL_2}(\hat{Q}^{i,j}_{m+r}\phi_m):m=1,\dots,l\}.
	\end{eqnarray*}
	Now, there are monomials 
	$P_{i,j}, \bar{P}_{j,i}, P^{(1)}_{i,j}, \bar{P}_{j,i}^{(1)}, P^{(2)}_{i,j}, 
	\bar{P}_{j,i}^{(2)}, P^{(3)}_{i,j}, \bar{P}^{(3)}_{i,j}$ such that 
	\begin{equation*}
		S_{i,j} = P_{i,j}F_i - \bar{P}_{i,j}F_j 
	\end{equation*}
	\begin{equation*}
		S^{(1)}_{i,j} = P^{(1)}_{i,j}F_i - \bar{P}^{(1)}_{i,j}F_j
	\end{equation*}
	\begin{equation*}
		S^{(2)}_{i,j} = P^{(2)}_{i,j}F_i - \bar{P}^{(2)}_{i,j}\phi_j
	\end{equation*}
	\begin{equation*}
		S^{(3)}_{i,j} = P^{(3)}_{i,j}\phi_i - \bar{P}^{(3)}_{i,j}\phi_j
	\end{equation*}
	Further, these monomials only depend on the initial terms of the $F_k, \phi_k$ involved on the right 
	hand sides of the above equations taken with respect to the appropriate ordering (i.e., the one 
	induced by $\LL_1$, or $\LL_2$). 
	Consider the following system of equations in variables 
	$y=(y_1,\dots,y_r)$, $z=(z^{1,2}_1,\dots,z^{r-1,r}_r)$, $w=(w_1^1,\dots,w_{r-s}^s)$, 
	$\tilde{y} = (\tilde{y}_1, \dots, \tilde{y}_l)$, $\bar{z} = (\tilde{z}^{1,1}_1, \dots, \tilde{z}^{s,l}_{s+l})$, 
	$\bar{y} = (\bar{y}_1, \dots, \bar{y}_l)$, $\bar{z} = (\bar{z}^{1,1}_1, \dots, \bar{z}^{s,l}_{s+l})$, 
	$\hat{z}=(\hat{z}^{1,2}_{1}, \dots, \hat{z}^{l-1,l}_{s+l})$, $\bar{w} = (\bar{w}_1^1, \dots, \bar{w}_{l-m}^{s+l})$.
	\begin{equation}
		\label{eq:sysflat}
		\begin{cases}
			\displaystyle{P_{i,j}(x)y_i-\bar{P}_{j,i}(x)y_j-\sum_{m=1}^rz^{i,j}_my_m=0} &{~}\\
			\displaystyle{P_{i,j}^{(1)}(x)y_i-\bar{P}_{j,i}^{(1)}(x)\tilde{y}_j-\sum_{m=1}^s\tilde{z}^{i,j}_m y_m - \sum_{m = 1}^l \tilde{z}^{i,j}_{m+s} \tilde{y}_m=0}  &{~}\\
			\displaystyle{P_{i,j}^{(2)}(x)y_i-\bar{P}_{j,i}^{(2)}(x)\bar{y}_j-\sum_{m=1}^s\bar{z}^{i,j}_m y_m - \sum_{m = 1}^l \bar{z}^{i,j}_{m+s} \bar{y}_m=0}  &{~}\\
			\displaystyle{P_{i,j}^{(3)}(x)\bar{y}_i-\bar{P}_{j,i}^{(3)}(x)\bar{y}_j-\sum_{m=1}^s\hat{z}^{i,j}_m y_m - \sum_{m = 1}^l \hat{z}^{i,j}_{m+s} \bar{y}_m=0}  &{~}\\
			\displaystyle{y_{s+p}-\sum_{q=1}^sw_p^qy_q}=0 &{~}\\
			\displaystyle{\bar{y}_{s+p}-\sum_{q=1}^s\bar{w}_p^qy_q} - \sum_{q=1}^m\bar{w}_p^{q+s} =0 &{~}
		\end{cases}
	\end{equation}
	This system has a solution in convergent power series 
	$\{F_i,Q^{i,j}_m, \tilde{Q}^{i,j}_m, \bar{Q}^{i,j}_m, H_p^q, \phi_i, \hat{Q}^{i,j}_m, \bar{H}_p^q\}$, and hence
	by \cite[Theorem 1.10]{Ar2}, for every $\mu \in \N$, an algebraic power series solution
	$\{G_i,R^{i,j}_m, \tilde{R}^{i,j}_m, \bar{R}^{i,j}_m, K_p^q, \psi_i, \hat{R}^{i,j}_m, \bar{K}_p^q\}$ such that
	\begin{eqnarray*}
		j^{\mu} G_i = j^{\mu}F_i\\
		j^{\mu} R^{i,j}_m = j^{\mu} Q^{i,j}_m \\
		j^{\mu} \tilde{R}^{i,j}_{m} = j^{\mu} \tilde{Q}^{i,j}_{m} \\
		j^{\mu} \bar{R}^{i,j}_{m} = j^{\mu} \bar{Q}^{i,j}_{m} \\
		j^{\mu} \hat{R}^{i,j}_{m} = j^{\mu} \hat{Q}^{i,j}_{m} \\ 
		j^{\mu} \psi_i = j^{\mu} \phi_i 
	\end{eqnarray*}
	for all allowable values of the indices. 

	Taking $\mu_0$ sufficiently large, so as to satisfy all the inequalities on the initial exponents
	with respect to both orderings (i.e., the ordering corresponding to $\LL_1$ and the one corresponding to $\LL_2$) 
	, by Theorem~\ref{thm:Becker}, the $G_i$ form a standard basis of the ideal 
	$\hat{I} = (G_1, \dots, G_r)$. In particular, the set $\{G_1, \dots, G_r\}$ contains representatives of 
	all the vertices of $\NN_{\LL_1}(\hat{I})$ (see Remark~\ref{rem:st-bases}(1)). Since by 
	construction, $\inexp_{\LL_1}{G_i} = \inexp_{\LL_1}{F_i}$ for all $i$, it follows that 
	$\NN_{\LL_1}(\hat{I}) = \NN_{\LL_1}(I)$. 
	Thus, $\NN_{\LL_1}(\hat{I})=D_1\times\N^{n-k}$ and so $\K\{x\}/\hat{I}$ is
	$\K\{\tilde{x}\}$-flat, by Proposition~\ref{prop:flatness-diagram}.
	Note also that $\hat{I}$ is, in fact, generated by $\{G_1,\dots,G_s\}$, since the
	remaining generators $G_{s+1},\dots,G_r$ are combinations of the former,
	by~\eqref{eq:sysflat}.

	The equality of the diagrams $\NN_{\LL_1}(\hat{I}) = \NN_{\LL_1}(I)$ implies, by
	Proposition \ref{prop:flatness-diagram} and Remark \ref{rem:CM-flat} that 
	$\K\{x\}/\hat{I}$ is Cohen-Macaulay. 
	Also, Lemma \ref{lem:HS-diagram-complement} implies that, 
	\begin{equation}
		\label{eq:dim11}
		\dim_\K\frac{\K\{x\}}{I+\nn_{\LL_1,\eta+1}}\ =\ \dim_\K\frac{\K\{x\}}{\hat{I}+\nn_{\LL_1,\eta+1}}\qquad\mathrm{for\ all\ }\eta\in\N\,.
	\end{equation}
	The equality $H_{I} = H_{\hat{I}}$ now follows from  
	\cite[Corollary 5.3]{JP}, the $\K\{\tx\}$-flatness of 
	$\K\{x\}/\hat{I}$, and \cite[Proposition 6.7]{JP} exactly as in the proof of 
	\cite[Theorem 7.3]{JP} and \cite[Theorem 8.1]{JP}. 
	Let $\hat{J} = (\psi_1, \dots, \psi_m)$. Then, by a similar argument in 
	terms of the ordering corresponding to $\LL_2$, $\K\{x\}/(\hat{I} + \hat{J})$ is Cohen-Macaulay 
	and $(\hat{I} + \hat{J})$ has 
	the same Hilbert-Samuel function as $(I + J)$. This implies 
	that $\dim \K\{x\}/(I + J) = \dim \K\{x\}/(\hat{I} + \hat{J})$. 
	Finally, by \cite[Theorem B.8.11]{GLS}
	the map $\psi : \hat{X}_0 \rightarrow \K^n_0$ defined by 
	$\psi = (\psi_1, \dots, \psi_m)$ is flat at zero. 
	\qed.

In the case when the germ $X$ is already Nash, it is not necessary to approximate the domain.  
\begin{corollary}
	\label{cor:cmaprox}
	Let $X$ be a $\K$-analytic subspace of $\K^n$. Suppose that $0 \in X$, $X_0$ is a Nash germ,
	the local ring 
	$\OO_{X,0}$ is Cohen-Macaulay, and that $\OO_{X,0} = \K\{x\}/I$. Also, let $\phi=(\phi_1, \dots, \phi_m): X \rightarrow \K^m$ 
	with $\phi(0) = 0$ be a $\K$-analytic mapping which is flat at 0. Then for some $\mu_0 \in \N$, and 
	each $\mu \geq \mu_0$ there is  
	a Nash map 
	$\psi =(\psi_{1}, \dots \psi_{m}) : X \rightarrow \K^m$ which is flat at 0 and such that 
	$j^{\mu} \phi_{i} = j^{\mu} \psi_{i}$ for all $i$. Further, if $J = (\phi_{1}, \dots, \phi_{m}) \subseteq \K\{x\}$, and 
	$\hat{J} = (\psi_{1}, \dots, \psi_{m}) \subseteq \K\langle x \rangle \subseteq \K\{x\}$, then we 
	have $H_{I + J} = H_{I + \hat{J}}$.
\end{corollary}
In this case there is no need to approximate the generators of $I$ because they are already algebraic power 
series. They appear as coefficients in the system of equations for 
approximating $\phi$. As in the proof of Theorem \ref{thm:cmapprox}, 
\cite[Theorem 1.10]{Ar2} will yield the required approximating map $\psi$.

\subsection{Preservation of tangent cones} 
If $X \subset \C^n$ is a complex analytic space, and $X_p$ is the germ of $X$ at $p \in X$, then
\emph{Whitney's tangent cone}(cf. \cite{Wh}) $C(X_p)$ of $X_p$ is defined as follows: 
$x \in \C^n$ belongs to $C(X_p)$ if and only if there is a sequence of points $\{p_{x,i}\}_{i = 1}^{\infty} 
\subseteq X$ and a sequence of complex numbers $\{a_{x,i}\}_{i = 1}^{\infty}$ such that 
$a_{i,x}(p - p_{i, x}) \rightarrow x$ as $i \rightarrow \infty$.  

Let $x = (x_1, \dots, x_n)$ and $F \in \C\{x\}$. Then $F$ may be written 
as the sum of homogeneous polynomials: $F = \sum_{i = 1}^{\infty} \sum_{|\beta| = i} f_{\beta} x^{\beta}$. 
Set $F^{(i)} = \sum_{|\beta| = i} f_{\beta} x^{\beta}$. If $i^* = \min \{i \in \N : F^{(i)} \neq 0\}$ then 
$\ins(F) = F^{(i^*)}$ is called the \emph{initial form} of $F$. If $I$ is an ideal in 
$\C\{x\}$ then $\ins(I) = \{\ins(F) : F\in I\}$ is called the \emph{ideal of initial forms} of $I$.
If $\OO_{X, p} = \C\{x\}/I$ then by
\cite[Theorem 10.6]{Wh}, $C(X_p) =\{x \in \C^n : \ins(F) (x) = 0 \text{ for all } F \in I\}$.  

\begin{lemma}
	\label{lem:initforms2}
	If $I \subset \C\{x\}$ is an ideal, and $F_1, \dots, F_s$ is a standard basis for 
	$I$ then $\ins(F_1), \dots, \ins(F_s)$ generate $\ins(I)$. 
\end{lemma}
\begin{proof}
	If $G \in I$, then by Remark \ref{rem:ssrep}
	\begin{equation*}
		G = \sum_{i = 1}^{s} Q_i F_i \text{ for some }Q_i \in \C\{x\}, 1 \leq i \leq s,  
	\end{equation*}
	where $\inexp(G) \leq \min{ \inexp(Q_i F_i)}$. Now, let $S = \{i : |\inexp(Q_i F_i)| = |\inexp(G)|\}$. 
	Then, $|\inexp(G)| \leq \min_i\{|\inexp(Q_i F_i)|\}$, and $\ins(F_i Q_i) 
	= \ins(F_i) \ins(Q_i)$ imply that
	\begin{equation*}
		\ins(G) = \sum_{i \in S} \ins(Q_i) \ins(F_i).
	\end{equation*}
\end{proof}

\begin{lemma}
	\label{lem:initforms3}
	If $I = (F_1, \dots, F_k)$ is an ideal in $\C\{x\}$, $\mu$ 
	is an integer such that $\mu > \max \{|\beta| : \beta \in \verts(I)\}$, 
	and $J \in U^{\mu}(I)$, then $\NN(I) = \NN(J)$ implies that 
	$\ins(I) = \ins(J)$. 
\end{lemma}
\begin{proof}
	Let $G_1, \dots, G_k$ be the generators of $J$. Also, 
	let $\hat{G}_1, \dots \hat{G}_s$ be representatives of 
	the vertices of $\NN(J)$. Then, there exist $Q^i_j \in \C\{x\}$, 
	$1 \leq i \leq k$, $1 \leq j \leq s$,  
	such that, 
	\begin{equation*}
		\hat{G}_j = \sum_{i = 1}^{k} Q_j^i G_i, \text{ for } 1\leq j \leq s. 
	\end{equation*}
	Let $\hat{F}_j = \sum_{i = 1}^{k} Q_j^i F_i$ for $1\leq j \leq s$. 
	Now, if $\mu > \max \{|\beta| : \beta \in \verts(I)\}$ then 
	for each $1 \leq j \leq s$, 
	\begin{equation*}
		j^{\mu} \hat{F}_j = j^{\mu} ( \sum_{i = 1}^{k} Q_j^i j^{\mu}F_i) 
		= j^{\mu} (\sum_{i = 1}^{k} Q_j^i j^{\mu} G_i) = j^{\mu} \hat{G}_j
	\end{equation*}
	implies that $\inexp(\hat{F}_j) = \inexp(\hat{G}_j)$. Therefore, 
	$\hat{F}_1, \dots, \hat{F}_s$ is a standard basis for $I$, and 
	by Lemma \ref{lem:initforms2}, $\ins(I) = \ins(J)$. 
\end{proof}

The corollaries to Theorems \ref{thm:ciapprox} and \ref{thm:cmapprox}, below,  
follow immediately from Lemma \ref{lem:eqdiags}, Lemma \ref{lem:initforms3} above, and
\cite[Theorem 10.6]{Wh}. 
\begin{corollary}
	\label{cor:cicones}
	Let $X$ be a $\K$-analytic subspace of $\K^n$. Suppose that $0 \in X$ and 
	the local ring $\OO_{X, 0}$ is a complete intersection. Also, let 
	$\phi = (\phi_1, \dots, \phi_m): X \rightarrow \K^n$ be a $\K$-analytic mapping
	that is flat at zero. Then there exists $\mu_0 \in \N$ such that for each 
	$\mu \geq \mu_0$, every map $\psi = (\psi_1, \dots, \psi_m): X \rightarrow \K^n$
	that satisfies $j^{\mu}\psi = j^{\mu}\phi$ has the following properties: 
	\begin{itemize}
		\item[(i)] $\psi$ is flat at zero.
		\item[(ii)] $C(\psi^{-1}(0)_0) = C(\phi^{-1}(0)_0)$.  
	\end{itemize}
\end{corollary}
\begin{corollary}
	\label{cor:cmcones}
	Let $X$ be a $\C$-analytic subspace of $\C^n$. Suppose that $0 \in X$ and the local ring 
	$\OO_{X,0}$ is Cohen-Macaulay and that $\OO_{X,0} = \C\{x\}/I$ where 
	$I = (F_1, \dots, F_s)$. Also, let $\phi=(\phi_1, \dots, \phi_m): X \rightarrow \C^m$ 
	with $\phi(0) = 0$ be a $\C$-analytic mapping which is flat at 0. Then there exists $\mu_0 \in \N$, such that 
	for each $\mu \geq \mu_0$ there exist:  
	\begin{itemize}
		\item[(a)] A Nash, Cohen-Macaulay germ $\hat{X}_0 \subseteq \C^n_0$ with $\OO_{\hat{X}, 0} = \C\{x\}/\hat{I}$, where
			$\hat{I} = (G_{1}, \dots, G_{s})$, for some $G_{i} \in \C \langle x \rangle$ that satisfy
			$j^{\mu}G_{i} = j^{\mu} F_i$ for all $i$, such that $C(\hat{X}_0) = C(X_0)$.  
		\item[(b)] A Nash map 
			$\psi =(\psi_{1}, \dots \psi_{n}) : \hat{X} \rightarrow \K^m$ which is flat at 0 with 
			$j^{\mu} \phi_{i} = j^{\mu} \psi_{i}$ for all $i$ for which 
			$C(\psi^{-1}(0)_0) = C(\phi^{-1}(0)_0)$. 
	\end{itemize}
\end{corollary}

\bibliographystyle{amsplain}

\end{document}